\documentclass[reqno,centertags, 12pt]{amsart}
\usepackage{amsmath,amsthm,amscd,amssymb}
\usepackage{latexsym}
\newcommand{\bbR}{{\mathbb{R}}}

\newcommand{\bdone}{{\boldsymbol{1}}}

\newcommand{\lb}{\label}
\newcommand{\f}{\frac}

\newcommand{\ti}{\tilde  }

\newcommand{\tr}{\text{\rm{tr}}}

\newcommand{\s}{\text{\rm{s}}}

\newcommand{\supp}{\text{\rm{supp}}}

\newcommand{\bi}{\bibitem}

\newcommand{\beq}{\begin{equation}}
\newcommand{\eeq}{\end{equation}}
\newcommand{\ba}{\begin{align}}
\newcommand{\ea}{\end{align}}
\newcommand{\veps}{\varepsilon}

\let\det=\undefined\DeclareMathOperator{\det}{det}

\newcounter{smalllist}
\newenvironment{SL}{\begin{list}{{\rm\roman{smalllist})}}{%
\setlength{\topsep}{0mm}\setlength{\parsep}{0mm}\setlength{\itemsep}{0mm}%
\setlength{\labelwidth}{2em}\setlength{\leftmargin}{2em}\usecounter{smalllist}%
}}{\end{list}}

\DeclareMathOperator{\Ima}{Im}

\allowdisplaybreaks
\numberwithin{equation}{section}

\newtheorem{theorem}{Theorem}[section]

\newtheorem*{p2.1}{Proposition 2.1}
\newtheorem{proposition}[theorem]{Proposition}
\newtheorem{lemma}[theorem]{Lemma}
\newtheorem{corollary}[theorem]{Corollary}
\theoremstyle{definition}
\newtheorem{example}[theorem]{Example}

\theoremstyle{remark}
\newtheorem*{remark}{Remark}

\newcommand{\abs}[1]{\lvert#1\rvert}

\begin{document}
\title{Monotone Jacobi Parameters and Non-Szeg\H{o} Weights}
\author[Y.~Kreimer, Y.~Last, and B.~Simon]
{Yury Kreimer$^1$, Yoram Last$^{1,3}$, and Barry Simon$^{2,3}$}

\thanks{$^1$ Institute of Mathematics, The Hebrew University,
91904 Jerusalem, Israel. E-mail: yuryk@math.huji.ac.il;
ylast@math.huji.ac.il. Supported in part by The Israel Science
Foundation (grant no. 1169/06)}

\thanks{$^2$ Mathematics 253-37, California Institute of Technology,
Pasadena, CA 91125, USA.
E-mail: bsimon@caltech.edu. Supported in part by NSF grant DMS-0140592}

\thanks{$^3$ Research supported in part
by Grants No.\ 2002068 and No.\ 2006483 from the United States-Israel Binational Science
Foundation (BSF), Jerusalem, Israel}

\date{November 7, 2007}

\begin{abstract} We relate asymptotics of Jacobi parameters to asymptotics of the
spectral weights near the edges.
Typical of our results is that for $a_n\equiv 1$, $b_n=-C n^{-\beta}$
($0<\beta< \f23)$, one has $d\mu(x)= w(x)\, dx$ on $(-2,2)$, and near $x=2$, $w(x)=e^{-2Q(x)}$
where
\[
Q(x)=\beta^{-1} C^{\f{1}{\beta}} \, \f{\Gamma(\f32)\Gamma(\f{1}{\beta}-\f12)(2-x)^{\f12 -\f{1}{\beta}}}
{\Gamma(\f{1}{\beta}+1)}\, (1+O((2-x)))
\]
\end{abstract}

\maketitle

\section{Introduction} \lb{s1}

Since the earliest days of the general theory of orthogonal polynomials on the real line (OPRL),
it has been known that a key role is played by the Szeg\H{o} condition \cite{Sz22a} that if
\begin{equation} \lb{1.1}
d\mu(x)=w(x)\, dx + d\mu_\s
\end{equation}
where $w$ is supported on $[-2,2]$ (we follow the spectral theorists' convention related to
$a_n \to 1$, $b_n\to 0$ rather than the $[-1,1]$ tradition in the OP literature), then
\begin{equation} \lb{1.2}
\int \log(w(x))(4-x^2)^{-\f12}\, dx > -\infty
\end{equation}
In this paper, we will examine asymptotics of $\log(w(x))$ for typical cases where \eqref{1.2}
fails. Recall \cite{Szb,FrB,Chi,OPUC1,Rice} that, given $\mu$, one can define monic orthogonal
and orthonormal polynomials $P_n(x,d\mu)$, $p_n(x,d\mu)$ and Jacobi parameters
$\{a_n,b_n\}_{n=1}^\infty$ by ($b_n$ real, $a_n >0$)
\begin{equation} \lb{1.3}
xp_n(x) =a_{n+1} p_{n+1}(x) + b_{n+1} p_n(x) + a_n p_{n-1}(x)
\end{equation}
and
\begin{equation} \lb{1.3a}
\|P_n\|=a_1 \cdots a_n
\end{equation}
Favard's theorem (see, e.g., \cite{OPUC1,Rice}) asserts a one-one correspondence between $\mu$'s
of compact but infinite support and bounded sets of $a_n$'s and $b_n$'s. Moreover, by Weyl's theorem,
if $a_n\to 1$, $b_n\to 0$, then the essential support of $d\mu$ is $[-2,2]$.

Roughly speaking, the boundary for \eqref{1.2} to hold is $a_n-1$, $b_n$ decaying faster than $O(n^{-1})$.
Explicitly, Killip and Simon \cite{KS} proved a conjecture of Nevai \cite{Nev92} that $\sum_{n=1}^\infty
(\abs{a_n-1} + \abs{b_n})<\infty \Rightarrow$ \eqref{1.2}, and there are examples of Pollaczek
\cite{Po49a,Po50a,Po56} where \eqref{1.2} fails because $\log(w(x))\sim (4-x^2)^{-\f12}$ near $x=\pm 2$
and $b_n=0$, $a_n=1 - Cn^{-1} + O(n^{-2})$.

Killip--Simon \cite{KS} discovered a relevant weaker condition than \eqref{1.2} they called the
quasi-Szeg\H{o} condition:
\begin{equation} \lb{1.4}
\int \log(w(x)) (4-x^2)^{\f12}\, dx > -\infty
\end{equation}
and they proved that
\begin{equation} \lb{1.5}
\text{\eqref{1.4}} + \sum_{x\in\supp(\mu)\setminus [-2,2]}
(\abs{x}-2)^{\f32} <\infty \Leftrightarrow \sum_{n=1}^\infty \, \abs{a_n-1}^2 +
\abs{b_n}^2 < \infty
\end{equation}
Our cases will include situations where \eqref{1.4} and \eqref{1.5} fail.

It is known (see \cite{Kh2001,Lub88,Lub91,MagVan,Sim157,Tom63}) that when $\sum_{n=1}^\infty
\abs{a_n-1}^2 + \abs{b_n}^2 =\infty$, $d\mu$ can stop having an a.c.\ component, so we
will need an additional condition. What we will use is

\begin{theorem}\lb{T1.1} If $a_n\to 1$, $b_n\to 0$, and
\begin{equation} \lb{1.6}
\sum_{n=1}^\infty\, \abs{a_{n+1}-a_n} + \abs{b_{n+1} -b_n} <\infty
\end{equation}
then \eqref{1.1} holds where $w(x)$ is continuous on $(-2,2)$ and strictly positive there.
Moreover, $d\mu_\s$ is supported on $\bbR\setminus (-2,2)$.
\end{theorem}

The continuum Schr\"odinger analog of this is a theorem of Weidmann \cite{We67}; for OPRL,
it is due to Dombrowski--Nevai \cite{DN} (see also \cite{Koo,GN,OPUC2}). Most references do
not discuss continuity of $w$ but it holds;
for example, it follows immediately from Theorem~1 of \cite{DN}, since $w$ can be obtained
as a uniform limit of continuous functions on any closed subinterval of $(-2,2)$.

In fact, we will focus on cases where $\{a_n\}$ and $\{b_n\}$ are monotone, so \eqref{1.6} is automatic.
Typical is
\begin{equation} \lb{1.7}
a_n\equiv 1 \qquad b_n = -C n^{-\beta}
\end{equation}
where, roughly speaking, we will prove $w(x)$ is singular at $x=2$ (i.e., the integral in
\eqref{1.4} diverges there) with
\begin{gather}
w(x) = e^{-2Q(x)} \lb{1.8} \\
Q(x) \sim C_1 (2-x)^{\f12 - \f1{\beta}} \lb{1.9}
\end{gather}
Indeed, in Section~\ref{s6}, we will obtain for \eqref{1.7} an asymptotic series for $Q(x)$
near $x=2$ up to terms of $O(\log(2-x))$; see \eqref{6.32}.

Our interest in these problems was stimulated by a recent paper of Levin--Lubinsky \cite{LL07}
and their related earlier works on non-Szeg\H{o} weights \cite{LL94,LL01}. They study the problem
inverse to ours, namely, going from $w$ (or $Q$) to $a_n,b_n$ (which they call $A_n,B_n$).
Unfortunately, they do not obtain even leading order asymptotics for $a_n,b_n$ if $Q(x)$ has
the form \eqref{1.9} but instead require
\begin{equation} \lb{1.10}
Q(x)\sim \exp_k (1-x^2)^{-\alpha}
\end{equation}
with $\exp_k(x)=\exp(\exp_{k-1}(x))$ and $\exp_1(x) = e^x$. We will obtain inverse
results to theirs in Section~\ref{s6}. We note that \cite{LL94} does have asymptotics on
the Rakhmanov--Mhaskar--Saff numbers when \eqref{1.9} holds and that their asymptotics
should be connected to asymptotics of $a_n, b_n$.

It is hard to imagine strict if and only if results on $Q(x)$ to $a_n,b_n$ since there will
typically be side conditions $(a_n,b_n$ monotone and/or convex in $n$ or $Q(x)$ convex) that may
not strictly carry over, but it is comforting (even with side conditions) to get results in
both directions. It would be interesting to show that \eqref{1.8} and \eqref{1.9} (with
extra conditions) lead to estimates on $a_n, b_n$ with $\abs{a_n-1} + \abs{b_n}=O(n^{-\beta})$.
We suspect, with analyticity assumptions on $Q$, that this might be accessible with
Riemann--Hilbert techniques.

Our key to going from $(a_n,b_n)$ to $(w,Q)$ is Carmona's formula that relates $d\mu$ to
the growth of $p_n(x)$, namely,

\begin{theorem}\lb{T1.2} If $p_n$ are the orthonormal polynomials for a measure $d\mu$, then
$d\nu^{(n)}\overset{w}{\longrightarrow}d\mu$ where
\begin{equation} \lb{1.11}
d\nu^{(n)}(x) = \f{dx}{\pi(a_n^2 p_n(x)^2 + p_{n-1}^2 (x))}
\end{equation}
\end{theorem}

The continuum analog of this result is due to Carmona \cite{Carm}. This theorem when $a_n=1$
is stated without proof in Last--Simon \cite{S263} and later (with proof) in Krutikov--Remling
\cite{KR} and Simon \cite{S-r47}. It implies:

\begin{corollary}\lb{C1.3} Suppose uniformly on some interval $[\alpha,\beta]$, we have for
strictly positive continuous functions $f_\pm(x)$ that
\begin{equation}\lb{1.12}
\begin{aligned}
\pi^{-1} f_-(x) &\leq \liminf (a_n^2 p_n(x)^2 + p_{n-1}(x)^2) \\
&\leq \limsup(a_n^2 p_n(x)^2 + p_{n-1}(x)^2) \leq \pi^{-1} f_+(x)
\end{aligned}
\end{equation}
Then $d\mu$ is purely absolutely continuous on $(\alpha,\beta)$ and
\begin{equation} \lb{1.13}
\f{1}{f_+(x)} \leq w(x) \leq \f{1}{f_-(x)}
\end{equation}
there. In particular, if \eqref{1.12} holds for each compact interval $[\alpha,\beta]$ in $(x_0,2)$,
\begin{equation} \lb{1.14}
f_\pm(x) =\exp(2(g(x)\pm h(x)))
\end{equation}
then \eqref{1.8} holds with
\begin{equation} \lb{1.15}
\abs{Q(x)-g(x)}\leq h(x)
\end{equation}
\end{corollary}

\begin{proof} By Theorem~\ref{T1.1}, for any positive continuous function, $\eta(x)$, on $[\alpha,\beta]$
supported on $(\alpha,\beta)$, we have
\begin{equation} \lb{1.15x}
\int \f{\eta(x)}{\pi f_+ (x)}\, dx \leq \int \eta(x)\, d\mu(x) \leq \int \f{\eta(x)}{\pi f_-(x)}\, dx
\end{equation}
from which absolute continuity of $\mu\restriction (\alpha,\beta)$ and \eqref{1.13} are immediate.
This in turn implies \eqref{1.14} and \eqref{1.15}.
\end{proof}

Thus, we need to show $a_n^2 p_n^2 + p_{n-1}^2$ is bounded as $n\to\infty$, but with bounds that diverge
as $x\uparrow 2$. The difference equation is
\begin{align}
\binom{p_{n+1}}{a_{n+1} p_n}
&= \f{1}{a_{n+1}} \begin{pmatrix} x-b_{n+1} & -1 \\
a_{n+1}^2 & 0 \end{pmatrix}
\binom{p_n}{a_n p_{n-1}} \notag \\
&\equiv A_{n+1}(x) \binom{p_n}{a_n p_{n-1}} \lb{1.16}
\end{align}
Here
\begin{equation} \lb{1.17}
\det (A_n)=1 \qquad \tr(A_n) = x-b_n
\end{equation}

In a case like \eqref{1.7} where $b_n$ is negative and monotone increasing, a fundamental object is the
turning point, the integer, $N(x)$, with
\begin{alignat}{2}
x-b_n &\geq 2 \qquad && \text{if } n\leq N(x) \lb{1.18} \\
x-b_n &<2 \qquad && \text{if } n > N(x) \lb{1.19}
\end{alignat}
If $\gamma_n(x)$ is defined by $\gamma_n\geq 0$ and
\begin{equation} \lb{1.20}
x-b_n =2\cosh (\gamma_n(x)) \qquad (n\leq N(x))
\end{equation}
then one expects some kind of exponential growth as $\exp (\sum_{j=1}^n \gamma_j(x))$, and we will prove
that
\begin{equation} \lb{1.21}
\exp\biggl(\, \sum_{j=1}^N \gamma_j(x)\biggr) \leq p_N(x) \leq (N+1)
\exp\biggl(\, \sum_{j=1}^N \gamma_j(x)\biggr)
\end{equation}

As one expects, there is an intermediate region $N(x)\leq n\leq N_1(x)$ and an oscillatory region
$n\geq N_1(x)$. We will see that so long as one is willing to accept $O((b_{N+2}-b_{N+1})^{-1})$
errors (and they will typically be very small compared to $\exp(\sum_{j=1}^N \gamma_j(x))$), one can
actually take $N_1 =N+2$ (!) and use the method of proof for Theorem~\ref{T1.1} to control the region
$n\geq N_1$. Thus, the key will be \eqref{1.21} and we will get \eqref{1.15} where
\begin{equation} \lb{1.22}
g(x) =\sum_{j=1}^N \gamma_j(x)
\end{equation}
and
\begin{equation} \lb{1.23}
h(x) =O(\max(\log(N), \log ((b_{N+2}-b_{N+1})^{-1})))
\end{equation}

The discussion of turning points sounds like WKB---and the reader might wonder if one can't obtain our
result via standard WKB techniques. There is some literature on discrete WKB \cite{GS92,SV92a,SV92b,SV94},
but we have not seen how to apply them to this situation (for a different application to OPRL, see
\cite{GSVA}) or, because of a double $n\to\infty$, $x\to 2$ limit, how to use the continuum WKB theory
(on which there is much more extensive literature) to the continuum analog of our problem here. That said,
the current paper should be regarded as a WKB-like analysis.

In Section~\ref{s2}, we discuss the case $a_n\equiv 1$, $b_n < b_{n+1} <0$. In Section~\ref{s3}, we
discuss $b_n\equiv 0$, $a_n < a_{n+1} <1$. It is likely one could handle mixed $a_n,b_n$ cases with
more effort. In Section~\ref{s4}, we discuss some Schr\"odinger operators. Finally, in
Section~\ref{s6}, we discuss examples including \eqref{1.7} and \eqref{1.10}.

\medskip
It is a pleasure to thank Fritz Gesztesy, Uri Kaluzhny, and Doron Lubinsky for useful discussions.
B.~S. would like to thank Ehud de Shalit for the hospitality of the Einstein
Institute of Mathematics at the Hebrew University where some of this work was done.
Y.~L. would like to thank Matthias Flach for the hospitality of 
the Department of Mathematics at Caltech where some of this work was done.

\section{Monotone $b_n$} \lb{s2}

In this section, we will prove:

\begin{theorem}\lb{T2.1} Let $d\mu$ be the spectral measure associated with a Jacobi matrix
having
$a_n\equiv 1$ and
\[
b_n\leq b_{n+1} < 0 \qquad
b_n\to 0 \text{ as } n\to\infty
\]
Define $N(x)$ for $x$ in $(0,2)$ and near $2$ by \eqref{1.18}/\eqref{1.19} and $\gamma_n(x)$
by \eqref{1.20}. Then $d\mu$ is purely absolutely continuous on $(-2,2)$,
where $w = \frac{d\mu}{dx}$ is continuous
and nonvanishing on $(-2,2)$,
\begin{equation} \lb{2.1}
C_1 (x+2) \leq w(x)\leq C_2 (x+2)^{-1} \quad\text{for } x\in (-2,0]
\end{equation}
and on $(0,2)$,
\begin{equation} \lb{2.2}
w(x)=e^{-2Q(x)}
\end{equation}
where
\begin{equation} \lb{2.3}
\abs{Q(x)-g(x)}\leq h(x)
\end{equation}
where
\begin{equation} \lb{2.4}
g(x) =\sum_{j=1}^{N(x)} \gamma_j(x)
\end{equation}
and $h(x)$ is given by
\begin{equation} \lb{2.5}
e^{h(x)} =CN(x) (b_{N(x)+2}-b_{N(x)+1})^{-1} (2-x)^{\f12}
\end{equation}
for an explicit constant $C$ {\rm{(}}dependent on $\sup\abs{b_n}$ but not on $x${\rm{)}}.
\end{theorem}

\begin{remark} Typically, $h$ is much smaller than $g$. For example, if $b_n$ is given by \eqref{1.7},
$g(x) = O ((2-x)^{\f12 - \f{1}{\beta}})$ and $e^{h(x)} =O(N(x)^{2+\beta}(2-x)^{\f12})=
O((2-x)^{-(\f12 +\f{2}{\beta})})$, so $h(x)=O(\log(2-x)^{-1})$.
\end{remark}

As we explained in the introduction, we need to study the asymptotics of $p_n(x)$ as $x\uparrow 2$ with
some uniformity in $n$. Given that $a_n\equiv 1$,
\begin{gather}
p_{n+1}(x)=(e^{\gamma_{n+1}} + e^{-\gamma_{n+1}}) p_n(x) -p_{n-1}(x) \lb{2.6} \\
p_{-1}(x)=0 \qquad p_0(x) =1 \lb{2.7}
\end{gather}
which suggests we define for $n\leq N(x)$,
\begin{equation} \lb{2.8}
\psi_n(x)=e^{-\sum_{j=1}^n \gamma_j} p_n(x)
\end{equation}
so $\psi_n$ obeys
\begin{gather}
\psi_{n+1}(x) = (1+e^{-2\gamma_{n+1}})\psi_n -e^{-(\gamma_n + \gamma_{n+1})} \psi_{n-1} \lb{2.8a} \\
\psi_{-1}(x) =0 \qquad \psi_0(x)=1 \lb{2.9}
\end{gather}

\begin{lemma}\lb{L2.2}
For $0\leq n < N(x)$,
\begin{equation} \lb{2.10}
\psi_{n+1} \geq \psi_n
\end{equation}
In particular,
\begin{equation} \lb{2.11}
\psi_n(x) \geq 1
\end{equation}
\end{lemma}

\begin{proof} As a preliminary, we note that $b_n\leq b_{n+1}$ implies $x-b_n \geq x-b_{n+1}$, so
\begin{equation} \lb{2.12}
0\leq \gamma_{n+1}\leq \gamma_n
\end{equation}

By \eqref{2.8a},
\begin{align}
(\psi_{n+1}-\psi_n) &= e^{-2\gamma_{n+1}} \psi_n -e^{-(\gamma_n +\gamma_{n+1})}\psi_{n-1} \notag \\
&= e^{-2\gamma_{n+1}} (\psi_n-\psi_{n-1}) + e^{-\gamma_{n+1}}(e^{-\gamma_{n+1}}-e^{-\gamma_n})
\psi_{n-1} \lb{2.13}
\end{align}
For $n=0$, $\psi_n-\psi_{n-1} =1\geq 0$ and $\psi_{n-1}=0\geq 0$. By \eqref{2.13} and \eqref{2.12}
(which implies $e^{-\gamma_{n+1}}-e^{-\gamma_n}\geq 0$), we see inductively that $\psi_{n+1} -\psi_n
\geq 0$, and so, $\psi_{n+1}\geq \psi_n\geq 0$, proving \eqref{2.10}.
\end{proof}

\begin{lemma}\lb{L2.3} Define for $n=0,1,2,\dots,N(x) - 1$,
\begin{equation} \lb{2.14}
W_n =e^{\gamma_{n+1}} \psi_n - e^{-\gamma_n} \psi_{n-1}
\end{equation}
Then
\begin{equation} \lb{2.15}
W_n\leq e^{\gamma_{n+1}}
\end{equation}
\end{lemma}

\begin{proof} $W_0=e^{\gamma_1} \leq e^{\gamma_1}$, starting an inductive proof of \eqref{2.15}.
By \eqref{2.8a},
\[
\psi_{n+1} =e^{-\gamma_{n+1}} W_n + e^{-2\gamma_{n+1}}\psi_n
\]
so
\begin{align}
W_{n+1} &= e^{(\gamma_{n+2} -\gamma_{n+1})} (W_n + e^{-\gamma_{n+1}}\psi_n) -e^{-\gamma_{n+1}} \psi_n \notag \\
&= e^{(\gamma_{n+2} -\gamma_{n+1})} W_n + e^{-\gamma_{n+1}} (e^{(\gamma_{n+2}-\gamma_{n+1})}-1) \psi_n \lb{2.16a} \\
&\leq e^{(\gamma_{n+2}-\gamma_{n+1})} W_n
\end{align}
since \eqref{2.12} implies $e^{\gamma_{n+2}}\leq e^{\gamma_{n+1}}$ and $\psi_n\geq 0$,
 $(e^{\gamma_{n+2}-\gamma_{n-1}}-1)\psi_n \leq 0$.
Thus, $W_n\leq e^{\gamma_{n+1}}$ implies $W_{n+1}\leq e^{\gamma_{n+2}}$ and
\eqref{2.15} holds inductively.
\end{proof}

\begin{lemma}\lb{L2.4}
For $n=0,1,2,\dots,N(x) - 2$,
\begin{equation} \lb{2.16}
\psi_{n+1}\leq 1 + \psi_n
\end{equation}
So, in particular, for $0\leq n < N(x)$,
\begin{equation} \lb{2.17}
\psi_n\leq n+1
\end{equation}
\end{lemma}

\begin{proof} By \eqref{2.14},
\begin{align*}
\psi_{n+1} &= e^{-\gamma_{n+2}} W_{n+1} + e^{-(\gamma_{n+1}+\gamma_{n+2})} \psi_n \\
&\leq 1+\psi_n
\end{align*}
since $e^{-\gamma_{n+2}} W_{n+1}\leq 1$ by \eqref{2.15} and $\gamma_j \geq 0$ implies
$e^{-(\gamma_{n+1}+\gamma_{n+2})}\leq 1$. This proves \eqref{2.16}, which inductively implies
\eqref{2.17}.
\end{proof}

We summarize with:

\begin{proposition}\lb{P2.5} For any $n$ with $1\leq n < N(x)$,
\begin{equation} \lb{2.18}
e^{\sum_{j=1}^n \gamma_j(x)} \leq p_n(x) \leq (n+1) e^{\sum_{j=1}^n \gamma_j(x)}
\end{equation}
In particular, if
\begin{equation} \lb{2.19}
\eta_n(x) = p_{n-1}(x)^2 + p_n(x)^2
\end{equation}
then
\begin{equation} \lb{2.20}
e^{2\sum_{j=1}^n \gamma_j(x)} \leq \eta_n(x) \leq 2(n+1)^2 e^{2\sum_{j=1}^n \gamma_j(x)}
\end{equation}
\end{proposition}

\begin{proof}
\eqref{2.18} is an immediate consequence of \eqref{2.8}, \eqref{2.11} and \eqref{2.17}.
\end{proof}

Suppose $x\in (0,2)$. For $n >N(x)$, define $\kappa_n(x)$ by $0\leq \kappa_n< \f{\pi}{2}$ and
\begin{equation} \lb{2.21}
x-b_n =2\cos\kappa_n(x)
\end{equation}
so $0>b_{n+1}\geq b_n$ implies
\[
0\leq \kappa_n\leq \kappa_{n+1}
\]
and $b_n\to 0$ implies
\begin{equation} \lb{2.22}
\kappa_n\to\kappa_\infty =\cos^{-1} (\tfrac{x}{2})
\end{equation}

For later reference, we note
\begin{equation} \lb{2.23}
\sin(\kappa_\infty) = (1-(\tfrac{x}{2})^2)^{\f12} = \tfrac12\, (4-x^2)^{\f12}
\end{equation}
So as $x\uparrow 2$,
\begin{equation} \lb{2.23a}
\kappa_\infty =(2-x)^{\f12} + O((2-x)^{\f32})
\end{equation}

We first present a matrix method following Kooman \cite{Koo} to control the region $[N(x)+2,\infty)$.
At the end, we will discuss an alternate method using scalar Pr\"ufer-like variables.

By \eqref{1.16}, for $n>N$\!, $A_n$ has eigenvalues $e^{\pm i\kappa_n}$. In fact,
\begin{equation} \lb{2.24}
\begin{pmatrix}
2\cos\kappa & -1 \\ 1 & 0
\end{pmatrix}
\begin{pmatrix} 1 \\ e^{\mp i\kappa}
\end{pmatrix}
= e^{\pm i\kappa}
\begin{pmatrix}
1 \\ e^{\mp i\kappa}
\end{pmatrix}
\end{equation}
so if
\begin{equation} \lb{2.25}
Y(\kappa) = \begin{pmatrix}
1 & 1 \\ e^{-i\kappa} & e^{i\kappa} \end{pmatrix}
\end{equation}
and
\begin{equation} \lb{2.25a}
V(\kappa) = \begin{pmatrix}
e^{i\kappa} & 0 \\
0 & e^{-i\kappa} \end{pmatrix}
\end{equation}
then
\begin{equation} \lb{2.26}
A_n(x) = Y(\kappa_n) V(\kappa_n) Y(\kappa_n)^{-1}
\end{equation}

Next, notice that
\begin{equation} \lb{2.27}
Y(\kappa)^{-1} = \f{1}{2i\sin \kappa}\begin{pmatrix}
e^{i\kappa} & -1 \\
-e^{-i\kappa} & -1 \end{pmatrix}
\end{equation}
Following Kooman \cite{Koo}, we write for $n > \ell > N(x)$,
\begin{align}
T_n(x) &\equiv A_n \cdots A_{\ell+1} \lb{2.28} \\
&= Y(\kappa_n) V_n Y(\kappa_n)^{-1} Y(\kappa_{n-1}) V_{n-1} \cdots Y(\kappa_{\ell+1})^{-1} \notag
\end{align}
and since $\|V_n(\kappa)\|=1$,
\begin{equation} \lb{2.29}
\|T_n\|\leq \|Y(\kappa_n)\|\, \|Y(\kappa_{\ell+1})^{-1}\| \prod_{j=\ell+1}^{n-1}
\|Y(\kappa_{j+1})^{-1} Y(\kappa_j)\|
\end{equation}

This prepares us for two critical estimates:

\begin{lemma}\lb{L2.6} We have
\begin{equation} \lb{2.30}
\|Y(\kappa_{j+1})^{-1} Y(\kappa_j)\| \leq 1 + \f{\abs{e^{i\kappa_{j+1}}-e^{i\kappa_j}}}
{\sin(\kappa_{j+1})}
\end{equation}
so, in particular,
\begin{equation} \lb{2.31}
\|Y(\kappa_{j+1})^{-1} Y(\kappa_j)\| \leq 1 + \f{\abs{\kappa_{j+1} -\kappa_j}}{\sin (\kappa_j)}
\end{equation}
\end{lemma}

\begin{proof} By \eqref{2.25} and \eqref{2.27},
\begin{equation} \lb{2.32}
Y(\kappa_{j+1})^{-1} Y(\kappa_j)-\bdone = \f{1}{2\sin(\kappa_{j+1})}
\begin{pmatrix}
e^{-i\kappa_{j+1}}-e^{-\kappa_j} & e^{i\kappa_{j+1}} -e^{i\kappa_j} \\
e^{-i\kappa_j} -e^{-i\kappa_{j+1}} & e^{i\kappa_j} - e^{i\kappa_{j+1}}
\end{pmatrix}
\end{equation}

If $A=(a_{ij})$ is a $2\times 2$ matrix,
\begin{align*}
\abs{\langle \varphi,A\psi\rangle} &\leq \max (\abs{a_{ij}}) (\abs{\varphi_1}+\abs{\varphi_2})
(\abs{\psi_1} + \abs{\psi_2}) \\
&\leq 2\max (\abs{a_{ij}}) (\abs{\varphi_1}^2)+\abs{\varphi_2}^2)^{\f12}
(\abs{\psi_1}^2 + \abs{\psi_2}^2)^{\f12}
\end{align*}
since $(\abs{x}+\abs{y})\leq\sqrt{2} (\abs{x}^2+\abs{y}^2)^{\f12}$, so
\[
\|Y(\kappa_{j+1})^{-1} Y(\kappa_j)-\bdone\| \leq \f{1}{\sin(\kappa_{j+1})}\,
\abs{e^{i\kappa_{j+1}} -e^{i\kappa_j}}
\]
which implies \eqref{2.30}.

\eqref{2.30} implies \eqref{2.31} since $\f{\pi}{2} > \kappa_{j+1} \geq \kappa_j$ implies
$\sin(\kappa_{j+1})\geq \sin(\kappa_j)$.
\end{proof}

\begin{remark} That \eqref{2.31} holds with a $1$ in front of $\abs{\kappa_{j+1}-\kappa_j}/
\sin(\kappa_j)$ is critical. Lest it seem a miracle of Kooman's method, we give an
alternate calculation at the end of this section.
\end{remark}

\begin{lemma}\lb{L2.7} We have that
\begin{equation} \lb{2.33}
\prod_{j=\ell+1}^\infty \biggl(1+\f{\abs{\kappa_{j+1}-\kappa_j}}{\sin(\kappa_j)}\biggr)
\leq \f{\kappa_\infty}{\kappa_{\ell+1}}\, \exp(\kappa_\infty e(\kappa_\infty))
\end{equation}
where
\begin{equation} \lb{2.34}
e(y) = \sup_{0 < x\leq y}\, \biggl(\f{1}{\sin(x)} - \f{1}{x}\biggr)
\end{equation}
\end{lemma}

\begin{remark} Since $\sin(x)=x-\f{x^3}{6}+O(x^5)$, $\f{1}{\sin(x)} =\f{1}{x} + \f{x}{6}
+ O(x^3)$ and since $\sin(x) <x$, we see $e(y)$ is finite and
\begin{equation} \lb{2.35}
e(y) =O(\tfrac{y}{6}) \qquad\text{as } y\downarrow 0
\end{equation}
\end{remark}

\begin{proof} We have
\begin{equation} \lb{2.36}
\f{1}{\sin(\kappa_j)} \leq \f{1}{\kappa_j} + e(\kappa_\infty)
\end{equation}
so, since $\kappa_{j+1}\geq\kappa_j$,
\begin{align}
1 + \f{\abs{\kappa_{j+1}-\kappa_j}}{\sin(\kappa_j)}
&\leq \f{\kappa_{j+1}}{\kappa_j} + (\kappa_{j+1} -\kappa_j) e(\kappa_\infty) \lb{2.37} \\
&\leq \f{\kappa_{j+1}}{\kappa_j} \, (1+(\kappa_{j+1}-\kappa_j)e(\kappa_\infty)) \lb{2.38} \\
&\leq \f{\kappa_{j+1}}{\kappa_j}\, \exp((\kappa_{j+1}-\kappa_j)e(\kappa_\infty)) \lb{2.39}
\end{align}
from which \eqref{2.33} is immediate if we note that $\kappa_\infty -\kappa_\ell\leq\kappa_\infty$.
\end{proof}

\begin{proof}[Proof of Theorem~\ref{T2.1}] By \eqref{2.29} and Lemmas~\ref{L2.6} and \ref{L2.7},
if $T_k(x)$ is the transfer matrix from $N(x)+2$ to $k>N(x)+2$, then uniformly in $k$,
\begin{equation} \lb{2.40}
\|T_n\| \leq 2 (\sin(\kappa_{N(x)+2}))^{-1} \, \f{\kappa_\infty}{\kappa_{N(x)+2}}\,
\exp (\kappa_\infty e(\kappa_\infty))
\end{equation}
where we also used $\|Y(\kappa_k)\|\leq 2$ and $\|Y(\kappa_{N(x)+2})^{-1}\|\leq 2/
2\sin(\kappa_{N(x)+2})$.

As $x\uparrow 2$, $\kappa_\infty\to 0$. Indeed, by \eqref{2.23a}, $\kappa_\infty = (2-x)^{\f12}
+O((2-x)^{\f32})$. Moreover, by the definition of $N(x)$,
\begin{equation} \lb{2.41}
x-b_{N+1} <2
\end{equation}
while
\begin{equation} \lb{2.42}
x-b_{N+2}=2\cos (\kappa_{N+2})
\end{equation}
so
\begin{equation} \lb{2.43}
2(1-\cos(\kappa_{N+2})) > b_{N+2}-b_{N+1}
\end{equation}

Since $N(x)\to\infty$, $b_{N(x)+2}\to 0$ so $\kappa_{N+2}(x)\to 0$ and \eqref{2.43} implies
\begin{equation} \lb{2.44}
\kappa_{N+2}(x)^2 > (1+o(1))(b_{N+2}-b_{N+1})
\end{equation}
Thus, in \eqref{2.40}, $[\kappa_{N(x)+2}\sin(\kappa_{N+2})]^{-1} \leq (1+o(1))
(b_{N+2}-b_{N+1})$ and \eqref{2.40} becomes
\begin{equation} \lb{2.45}
\sup_{n\geq N(x)+2}\, \|\ti T_n\| \leq C(2-x)^{\f12} (b_{N+2}-b_{N+1})^{-1}\equiv A(x)
\end{equation}
where now $\ti T_n$ transfers from $N-1$ to $n$
and we use the boundedness from $N-1$ to $N+2$.
Using
\begin{equation} \lb{2.46}
\|\ti T_n\|^{-2} (\abs{p_{n+1}}^2 + \abs{p_n}^2) \leq \abs{p_N}^2 + \abs{p_{N-1}}^2 \leq
\|\ti T_n^{-1}\|^2 (\abs{p_{n+1}}^2 + \abs{p_n}^2)
\end{equation}
and \eqref{2.20}, we obtain for all $n>N$\!,
\begin{equation} \lb{2.47}
C_1 A(x)^{-2} e^{2\sum_1^N \gamma_j(x)} \leq (\abs{p_n}^2 + \abs{p_{n+1}}^2) \leq
CA(x)^2 N(x)^2 e^{2\sum_1^N \gamma_j(x)}
\end{equation}
which, given Corollary~\ref{C1.3}, implies \eqref{2.2}--\eqref{2.4}.

In going from \eqref{2.46} to \eqref{2.47}, we used
\[
\det(\ti T_n)=1\Rightarrow \|\ti T_n^{-1}\| = \|\ti T_n\|
\]

We also need to control the region $x>-2$ with $2-x$ small. By replacing $x$ by $-x$ (and $p_n(x)$ by
$(-1)^n p_n(-x)$), this is the same as looking at $x+b_n$ with still $b_n<b_{n+1} <0$. We define
$\theta_n(x)$ by
\begin{equation} \lb{2.47a}
2\cos(\theta_n(x)) =x+b_n
\end{equation}
so
\begin{equation} \lb{2.47b}
\theta_1\geq\theta_2\geq\cdots\geq\theta_\infty = \kappa_\infty = (2-x)^{\f12} + O((2-x)^{\f32})
\end{equation}

As above, we have \eqref{2.30}, so
\begin{equation} \lb{2.47c}
\|Y(\theta_{j+1})^{-1} Y(\theta_j)\|\leq 1 + \f{\abs{\theta_{j+1}-\theta_j}}{\sin(\theta_{j+1})}
\end{equation}
but since $\theta_{j+1} <\theta_j$, we have
\begin{equation} \lb{2.47d}
1+\f{\abs{\theta_{j+1}-\theta_j}}{\abs{\theta_{j+1}}} = \f{\theta_{j+1}+(\theta_j-\theta_{j+1})}
{\theta_{j+1}} = \f{\theta_j}{\theta_{j+1}}
\end{equation}
and we find that, with $T_n$ being the transfer matrix from $1$ to $n$,
\begin{equation} \lb{2.47e}
\|T_n\|\leq \f{\theta_1}{\theta_\infty}\, 2\, \f{2}{2\sin(\theta_1)} \leq
\f{C}{\theta_\infty} \leq C(2-x)^{\f12} (1+o(1))
\end{equation}
This bound on the transfer matrix and Corollary~\ref{C1.3} yield \eqref{2.1}.
\end{proof}

\begin{remark} It might be surprising that \eqref{2.1} has $(x+2), (x+2)^{-1}$ rather than
$(x+2)^{\f12},  (x+2)^{-\f12}$ (because Carmona's bound relates $w(x)$ to $\|T_n\|^2$ and
$\sup\|T_n\|$ goes like $(2-x)^{\f12}$). Even in the free case, bounds from Carmona's formula
give the wrong behavior: $\sin(n\theta)+\sin^2((n+1)\theta)$ have oscillations that cause the
actual square root behavior in the free case, and bounds based only on $\|T_n\|$ lose that.
\end{remark}

That completes the proof of Theorem~\ref{T2.1}, the main result of this paper.
Here is an alternate approach to controlling $p_n$ for $n>N$,
using the complex quantities:
\begin{equation} \lb{2.48}
\Phi_n =p_n -e^{-i\kappa_n} p_{n-1}
\end{equation}
so, since $p_j$ is real,
\begin{align}
\sin (\kappa_n)\abs{p_{n-1}} &= \abs{\Ima (-\Phi_n)} \notag \\
&\leq \abs{\Phi_n} \lb{2.49}
\end{align}

By \eqref{2.21}, we have
\begin{equation} \lb{2.50}
p_{n+1} =(e^{i\kappa_{n+1}} + e^{-i\kappa_{n+1}}) p_n-p_{n-1}
\end{equation}
so
\begin{align}
\Phi_{n+1} &= e^{i\kappa_{n+1}}[p_n - e^{-i\kappa_{n+1}} p_{n-1}] \notag \\
&= e^{i\kappa_{n+1}} \Phi_n + e^{i\kappa_{n+1}}(e^{-i\kappa_n} -e^{-i\kappa_{n+1}}) p_{n-1} \lb{2.51}
\end{align}
Using \eqref{2.49},
\begin{equation} \lb{2.52}
\abs{\Phi_{n+1}} \leq \abs{\Phi_n} + \f{\abs{\kappa_n -\kappa_{n+1}}}{\sin(\kappa_n)}\, \abs{\Phi_n}
\end{equation}
and similarly, 
\begin{equation} \lb{2.53}
\abs{\Phi_{n+1}} \geq \abs{\Phi_n} - \f{\abs{\kappa_n-\kappa_{n+1}}}{\sin(\kappa_n)}\, \abs{\Phi_n}
\end{equation}

These replace \eqref{2.31} and imply, via Lemma~\ref{L2.7} and the analysis in \eqref{2.41},
that
\[
C_1 (2-x)^{-\f12} (b_{N+2}-b_{N+1}) \leq \f{\abs{\Phi_n}}{\abs{\Phi_{N+2}}} \leq
C(2-x)^{\f12} (b_{N+2}-b_{N+1})^{-1}
\]

Since
\[
\abs{\Phi_n}^2\leq \abs{p_n}^2 + \abs{p_{n-1}}^2
\]
and
\[
2\abs{\Phi_n}^2 \geq \sin^2 (\kappa_{n+1}) (\abs{p_n}^2 + \abs{p_{n-1}}^2)
\]
we can go from this to Theorem~\ref{T2.1}.

\section{Monotone $a_n$} \lb{s3}

In this section, we will consider
\begin{equation} \lb{3.1}
b_n\equiv 0 \qquad a_{n+1}\leq a_n\leq 1 \qquad a_n\to 1
\end{equation}
The weight will be symmetric, the measure purely absolutely continuous (i.e., no eigenvalues outside
$[-2,2]$), and so for non-Szeg\H{o} weights, the integral will diverge at both ends. Here is the main
result:

\begin{theorem}\lb{T3.1} Let $d\mu(x) = w(x)\,dx$ be the measure associated with
Jacobi parameters obeying \eqref{3.1}. For any
$x\in (-2,2)$, define $N(x)$ by
\begin{equation} \lb{3.2}
2a_n \leq \abs{x}\quad\text{for } n\leq N(x) \qquad
2a_n > \abs{x}\quad\text{for } n> N(x)
\end{equation}
and $\gamma_n(x)$ for $n\leq N(x)$ by
\begin{equation} \lb{3.3}
\f{\abs{x}}{a_n} =2\cosh (\gamma_n (x))
\end{equation}
Then
\begin{equation} \lb{3.4}
w(x)=e^{-2Q(x)}
\end{equation}
where
\begin{gather*}
\abs{Q(x)-g(x)} \leq h(x) \\ 
g(x) =\sum_{j=1}^{N(x)} \gamma_j(x)
\end{gather*}
and $h(x)$ is given by
\begin{equation} \lb{3.5}
e^{h(x)}=CN(x) (a_{N(x)+2} -a_{N(x)+1})^{-1}
\end{equation}
\end{theorem}

The proof will closely mimic the proof of Theorem~\ref{T2.1}, so we will only indicate the changes.
By symmetry, without loss, we can suppose $x>0$. The recursion relation becomes
\begin{equation} \lb{3.6}
p_{n+1}(x) = (e^{\gamma_{n+1}(x)} + e^{-\gamma_{n+1}(x)}) p_n(x) - \f{a_n}{a_{n+1}}\,
p_{n-1}(x)
\end{equation}
where we note, by \eqref{3.3}, that
\begin{equation} \lb{3.7}
\f{a_n}{a_{n+1}} = \f{\cosh(\gamma_{n+1}(x))}{\cosh(\gamma_n(x))}
\end{equation}

Define $\psi_n(x)$ by \eqref{2.8}, so \eqref{2.8a} becomes
\begin{equation} \lb{3.8}
\psi_{n+1}(x) = (1+e^{-2\gamma_{n+1}(x)}) \psi_n(x) - \f{a_n}{a_{n+1}}\,
e^{-(\gamma_n(x) + \gamma_{n+1}(x))} \psi_{n-1}(x)
\end{equation}
\eqref{2.9} still holds.

\begin{lemma}\lb{L3.2} $\psi_{n+1}\geq\psi_n$, so $\psi_n(x) \geq 1$ for $n\geq 0$.
\end{lemma}

\begin{proof} We still have \eqref{2.12}, and \eqref{2.13} becomes
\begin{equation} \lb{3.9}
\psi_{n+1} -\psi_n =e^{-2\gamma_{n+1}}(\psi_n -\psi_{n-1}) + e^{-\gamma_{n+1}}
\biggl(e^{-\gamma_{n+1}} - \f{a_n}{a_{n+1}}\, e^{-\gamma_n}\biggr) \psi_{n-1}
\end{equation}
Since $a_n\leq a_{n+1}$, $\f{a_n}{a_{n+1}}<1$, and so
\[
\f{a_n}{a_{n+1}} \, e^{-\gamma_n} \leq e^{-\gamma_n}\leq e^{-\gamma_{n+1}}
\]
Thus, by \eqref{2.8a}, $\psi_{n+1} -\psi_n \geq 0$ and $\psi_{n+1}\geq 0$ inductively.
\end{proof}

\begin{lemma}\lb{L3.3}
\begin{equation} \lb{3.10}
e^{\gamma_{n+2}}\leq e^{\gamma_{n+1}}\,\f{\cosh(\gamma_{n+2})}{\cosh (\gamma_{n+1})}
\end{equation}
\end{lemma}

\begin{proof} This is equivalent to
\begin{equation} \lb{3.11}
e^{\gamma_{n+2}+\gamma_{n+1}} + e^{\gamma_{n+2}-\gamma_{n+1}} \leq e^{\gamma_{n+2} + \gamma_{n+1}}
+ e^{\gamma_{n+1} - \gamma_{n+2}}
\end{equation}
so to $\gamma_{n+2} -\gamma_{n+1}\leq 0$, so to \eqref{2.12}.
\end{proof}

\begin{lemma}\lb{L3.4} Define
\begin{equation} \lb{3.12}
W_n=e^{\gamma_{n+1}} \psi_n -\f{a_n}{a_{n+1}}\, e^{-\gamma_n} \psi_{n-1}
\end{equation}
Then
\begin{equation} \lb{3.13}
W_n\leq e^{\gamma_{n+1}}
\end{equation}
\end{lemma}

\begin{proof} \eqref{3.13} holds for $n=0$ by \eqref{3.12} for $n=0$, so we can try an inductive
proof. The analog of \eqref{2.16a} is
\begin{equation} \lb{3.14}
W_{n+1} =e^{(\gamma_{n+2} -\gamma_{n+1})} W_n + e^{-\gamma_{n+1}}
\biggl(e^{(\gamma_{n+2}-\gamma_{n+1})} - \f{a_{n+1}}{a_{n+2}}\biggr) \psi_n
\end{equation}
By \eqref{3.7} and \eqref{3.10},
\[
e^{(\gamma_{n+2} -\gamma_{n+1})} - \f{a_{n+1}}{a_{n+2}} \leq 0
\]
so \eqref{3.14} says
\[
W_{n+1} \leq e^{(\gamma_{n+2} -\gamma_{n+1})} W_n \leq e^{\gamma_{n+2}}
\]
by induction.
\end{proof}

\begin{lemma}\lb{L3.5} $\psi_{n+1}\leq 1 +\psi_n$ so inductively, $\psi_n\leq n+1$.
\end{lemma}

\begin{proof} By \eqref{3.12} and \eqref{3.13},
\begin{align*}
\psi_{n+1} &= e^{-\gamma_{n+2}} W_{n+1} + \f{a_{n+1}}{a_{n+2}}\, e^{-\gamma_{n+2} - \gamma_{n+1}} \psi_n \\
&\leq 1 + \psi_n
\end{align*}
since $\f{a_{n+1}}{a_{n+2}}\leq 1$.
\end{proof}

If now
\begin{equation} \lb{3.15}
\eta_n(x) = p_{n-1}(x)^2 + a_{n}^2 p_n(x)^2
\end{equation}
then we have proven \eqref{2.20} for large $n$.

To control the region $n\geq N(x)+2$,
we use the scalar variable technique from the end of Section~\ref{s2}.
Define $\kappa_n$ for $n\geq N(x)+1$ by (recall $x>0$)
\begin{equation} \lb{3.16}
\f{x}{a_n} = 2\cos (\kappa_n(x))
\end{equation}
so $a_{n+1}\geq a_n$ implies
\begin{equation} \lb{3.16a}
\kappa_n(x) \leq \kappa_{n+1}(x)
\end{equation}
Define
\begin{equation} \lb{3.17}
\Phi_n =p_n -e^{-i\kappa_n} p_{n-1}
\end{equation}
Then
\begin{lemma}\lb{L3.6}
\begin{SL}
\item[{\rm{(i)}}]
\begin{equation} \lb{3.18}
\abs{p_{n-1}} \leq \f{\abs{\Phi_n}}{\sin(\kappa_n)}
\end{equation}
\item[{\rm{(ii)}}]
\begin{align}
\f{\abs{\Phi_{n+1}}}{\abs{\Phi_n}}
& \leq 1 + \f{\abs{e^{i\kappa_n}\cos(\kappa_n) - e^{i\kappa_{n+1}}\cos(\kappa_{n+1})}}
{\cos(\kappa_n)\sin(\kappa_n)} \lb{3.19} \\
&\leq 1 + \f{\abs{\kappa_{n+1} -\kappa_n}}{\f12 \sin(2\kappa_n)} \lb{3.20}
\end{align}
\end{SL}
\end{lemma}

\begin{proof} (i) This comes from $\abs{\Ima \Phi_n} =\sin (\kappa_n)(p_{n-1})$.

\smallskip
(ii) From
\[
p_{n+1} = (e^{i\kappa_{n+1}} + e^{-i\kappa_{n+1}}) p_n - \f{a_n}{a_{n+1}}\, p_{n-1}
\]
we obtain
\begin{equation} \lb{3.21}
\abs{\Phi_{n+1} -e^{i\kappa_{n+1}} \Phi_n} = \biggl| e^{i\kappa_n} - \f{a_n}{a_{n+1}}
\, e^{i\kappa_{n+1}}\biggr|\, p_{n-1}
\end{equation}

By \eqref{3.16},
\begin{equation} \lb{3.22}
\f{a_n}{a_{n+1}} = \f{\cos (\kappa_{n+1})}{\cos(\kappa_n)}
\end{equation}
so \eqref{3.21} and \eqref{3.18} imply \eqref{3.19}. This in turn implies \eqref{3.20} since
\begin{equation} \lb{3.23}
e^{i\kappa_n} \cos(\kappa_n) -e^{i\kappa_{n+1}} \cos(\kappa_{n+1}) = \tfrac12\,
(e^{2i\kappa_n} - e^{2i\kappa_{n+1}})
\end{equation}
\end{proof}

With this formula, we can mimic the proof of Theorem~\ref{T2.1} to complete the proof of
Theorem~\ref{T3.1}.

\section{Schr\"odinger Operators} \lb{s4}

In this section, we consider Schr\"odinger operators $H=-\f{d^2}{dx^2}+V(x)$ on
$L^2 ([0,\infty))$ where one places $u(0)=0$ boundary conditions. $H$ is unitarily
equivalent to multiplication by $E$ on $L^2 (\bbR,d\mu(E))$, where $d\mu$ is the conventional
spectral measure (see \cite{CL85,LG75,Mar}). If $u(x,E)$ obeys
\begin{equation} \lb{4.1}
-u'' + Vu=Eu \qquad u(0,E)=0,\,\, u'(0,E)=1
\end{equation}
then Carmona's formula \cite{Carm} takes the form
\begin{equation} \lb{4.2}
\f{\pi^{-1} dE}{(\abs{u(x,E)}^2 + \abs{u'(x,E)}^2)} \overset{w}{\longrightarrow} d\mu(E)
\end{equation}
In particular, if uniformly in compact subsets of $E\in (0,\infty)$,
\begin{align}
\exp(2(g(E)-h(E))) &\leq \liminf_{x\to\infty}\, (\abs{u(x,E)}^2 + \abs{u'(x,E)}^2) \notag \\
&\leq \limsup_{x\to\infty}\, (\abs{u(x,E)}^2 + \abs{u'(x,E)}^2) \notag \\
&\leq \exp(2(g(E)+h(E))) \lb{4.3}
\end{align}
then $d\mu$ is purely absolutely continuous on $(0,\infty)$, $d\mu(E)=e^{-2Q(E)}\, dE$, and
\begin{equation} \lb{4.4}
\abs{Q(E)-g(E)}\leq h(E)
\end{equation}

We want to assume the following conditions on $V$:
\begin{SL}
\item[(a)] $V$ is $C^1$ on $[0,\infty)$.
\item[(b)] $V$ is positive and strictly monotone decreasing on $[0,\infty)$. Indeed,
\begin{equation} \lb{4.5}
V'(x) <0
\end{equation}
\item[(c)]
\begin{equation} \lb{4.5a}
\lim_{x\to\infty}\, V(x) =0
\end{equation}
\end{SL}
Of course, the canonical example is
\begin{equation} \lb{4.6}
V(x)=(x+x_0)^{-\beta}
\end{equation}

Our main result in this section is:

\begin{theorem}\lb{T4.1}
Let $V$ obey {\rm{(a), (b), (c)}} so $d\mu(E)=e^{-2Q(E)}\, dE$. Define
for $E<V(0)$,
\[
N(E)=V^{-1}(E)
\]
so
\begin{equation}\lb{4.8}
\begin{alignedat}{2}
V(x) & >E \qquad &&\text{if }\, x<N(E) \\
V(x) & <E \qquad &&\text{if }\, x>N(E)
\end{alignedat}
\end{equation}
For $x<N(E)$, define
\begin{equation} \lb{4.9}
\gamma(x,E)=(V(x)-E)^{\f12}
\end{equation}
Then \eqref{4.4} holds where for $E<V(0)$,
\begin{equation} \lb{4.10}
g(E)=\int_0^{N(E)} \gamma(x,E)\, dx
\end{equation}
and for $E<V(0)$,
\begin{equation} \lb{4.11}
e^{h(E)} =C N(E)\, (V(N(E)) - V(N(E)+1))^{-1} E^{\f12}
\end{equation}
\end{theorem}

This proof will illuminate the proofs of the previous two sections.
We begin with an analysis of the region $x<N(E)$. We define 
\begin{equation} \lb{4.12}
\psi(x)=u(x,E) \exp\biggl( -\int_0^x \gamma(y,E)\, dy\biggr)
\end{equation}
and are heading towards
\begin{equation} \lb{4.13}
0\leq \psi'(x)\leq 1
\end{equation}

\begin{lemma}\lb{L4.2} For $0<E<V(0)$ and $x<N(E)$, we have
\begin{alignat}{2}
&\text{\rm{(a)}} \qquad && u'(x)\geq 1 \lb{4.14} \\
&\text{\rm{(b)}} \qquad && u(x) \geq x \lb{4.15}
\end{alignat}
\end{lemma}

\begin{proof} $u''=\gamma^2 u$, so $u'' >0$. This implies $u'(x)\geq u'(0)= 1$, and then $u(x) =
\int_0^x u'(y)\, dy\geq x$.
\end{proof}

\begin{lemma}\lb{L4.3} For $E<V(0)$ and $x<N(E)$,
\begin{equation} \lb{4.16}
\psi'(x)\geq 0
\end{equation}
\end{lemma}

\begin{proof} Let
\begin{equation} \lb{4.17}
f(x)=u'(x) -\gamma(x) u(x)
\end{equation}
so
\begin{equation} \lb{4.18}
\psi'(x) =f(x) \exp\biggl( -\int_0^x \gamma(y,E)\, dy \biggr)
\end{equation}
and \eqref{4.16} is equivalent to $f\geq 0$. Note that
\begin{align}
f'+ \gamma f &= u'' -\gamma u' -\gamma' u + \gamma u' - \gamma^2 u \notag \\
&= -\gamma' u \lb{4.19}
\end{align}
since \eqref{4.1} says
\begin{equation} \lb{4.20}
u'' =\gamma^2 u
\end{equation}
\eqref{4.5} implies
\begin{equation} \lb{4.21}
\gamma' (y)\leq 0
\end{equation}
so \eqref{4.19} says
\begin{equation} \lb{4.22}
\biggl( f \exp\biggl( \int_0^x \gamma(y)\, dy\biggr)\biggr)' \geq 0
\end{equation}
which, given $f(0)=1$, implies $f\geq 0$ and so $\psi' \geq 0$.
\end{proof}

\begin{lemma}\lb{L4.5} Let
\begin{equation} \lb{4.23}
W(x)=\psi'(x) + 2\gamma(x) \psi(x)
\end{equation}
Then $W'(x)\leq 0$ and so
\begin{equation} \lb{4.24}
\psi'(x)\leq 1
\end{equation}
\end{lemma}

\begin{proof} By \eqref{4.18},
\begin{equation} \lb{4.25}
\psi' + 2\gamma(x) \psi = (u' +\gamma(x) u) e^{-\int_0^x \gamma(y)\, dy}
\end{equation}
so
\begin{align}
W'(x) &= (u'' + \gamma u' + \gamma' u - \gamma u' -\gamma^2 u) e^{-\int_0^x \gamma(y)\, dy} \notag \\
&= \gamma' u e^{-\int_0^x \gamma(y)\, dy} \lb{4.26} \\
&\leq 0 \notag
\end{align}
by \eqref{4.21}. But $W(x=0)=\psi'(0)=1$, so
\begin{equation} \lb{4.27}
W(x) \leq 1
\end{equation}
and thus
\begin{equation} \lb{4.28}
\psi' =W-2\gamma \psi\leq 1
\end{equation}
\end{proof}

\begin{proposition}\lb{P4.6} If $E$ is such that $N(E) >1$, then
\begin{equation} \lb{4.29}
e^{-2V(0)} e^{2\int_0^{N(E)}\gamma(y)\, dy} \leq u(N(E))^2 + u'(N(E))^2 \leq (N(E)^2+1)
e^{2\int_0^{N(E)} \gamma(y)\, dy}
\end{equation}
\end{proposition}

\begin{proof} Since $\gamma(N(E))=0$,
\[
\psi' (N(E)) = u' (N(E)) e^{-\int_0^{N(E)} \gamma(y)\, dy}
\]
so $0\leq \psi'\leq 1$ and $\psi(0) = 0$ yield the upper bound in \eqref{4.29}.

For the lower bound, \eqref{4.15} implies $u(1)\geq 1$. So, since $\gamma(y)\leq\gamma(0)\leq V(0)$,
\begin{equation} \lb{4.31}
\psi(1)\geq e^{-V(0)}
\end{equation}
which, given that $\psi' >0$ and $N(E) >1$, implies
\[
u(N(E))\geq e^{-V(0)} e^{\int_0^{N(E)}\gamma(y)\, dy}
\qedhere
\]
\end{proof}

In the region $[N(E), N(E)+1]$, we note that since
\[
\biggl\| \begin{pmatrix} 1 & V(x)-E \\ 1 & 0
\end{pmatrix} \biggr\| \leq 1+\abs{E} + \abs{V(0)}
\]
the matrix form of the Schr\"odinger equation implies that if $C(x)=\abs{u(x)}^2 + \abs{u'(x)}^2$, then
\[
e^{-2(1+\abs{E}+V(0))\abs{x-y}} C(y) \leq C(x) \leq e^{2(1+\abs{E}+V(0))\abs{x-y}} C(y)
\]
giving a constant term in $e^{h(E)}$ in \eqref{4.11}.

Finally, in the region $[N(E)+1,\infty)$, we use the method of Appendix~2 of Simon \cite{S253}
(see also Hinton--Shaw \cite{HS}). Define for $x >N(E)$,
\begin{equation} \lb{4.34}
\kappa(x,E)=\sqrt{E-V(x)}
\end{equation}
and define
\begin{equation} \lb{4.35}
u_\pm (x) =\exp \biggl( \pm i \int_{N(E)}^x \kappa(y)\, dy\biggr)
\end{equation}
If
\begin{equation} \lb{4.36}
F(x)=\f{i}{2}\, V'(x) (E-V(x))^{-\f12}
\end{equation}
and if $a(x)$, $b(x)$ are defined by
\begin{align}
u(x)&=a(x) u_+(x) + b(x) u_-(x) \lb{4.38} \\
u'(x) &= a(x) u'_+(x) + b(x) u'_-(x) \lb{4.39}
\end{align}
then $u'' =-\kappa^2 u$ is equivalent to (see Problem~98 on p.~395 of \cite{RS3})
\begin{equation} \lb{4.40}
\binom{a(x)}{b(x)}' = M(x) \binom{a(x)}{b(x)}
\end{equation}
where
\begin{equation} \lb{4.41}
M(x)=w(x)^{-1} \begin{pmatrix}
-F(x) & u_-^2(x) F(x) \\
u_+^2 (x) F(x) & -F(x)
\end{pmatrix}
\end{equation}
with
\begin{align}
w(x) &= u'_+ (x) u_-(x) - u'_-(x) u_+(x) \notag \\
&= 2i\kappa(x) \lb{4.42}
\end{align}

\begin{proposition}\lb{P4.7} Let $M(x)$ be given by \eqref{4.41}. Then
\begin{equation} \lb{4.43}
\int_{N(E)+1}^\infty \|M(x)\|\, dx \leq \log\biggl( \f{\kappa(\infty, E)}{\kappa(N(E)+1,E)}\biggr)
\end{equation}
\end{proposition}

\begin{proof} Since $\abs{u_\pm} =1$,
\begin{align}
\|M(x)\| &\leq \abs{w(x)}^{-1} \biggl\| \begin{pmatrix}
\abs{F(x)} & \abs{F(x)} \notag \\
\abs{F(x)} & \abs{F(x)} \end{pmatrix} \biggr\| \notag \\
&= 2\abs{w(x)}^{-1} \abs{F(x)} \notag \\
&= -\tfrac12\, V'(x) (E-V(x))^{-1} \notag \\
&= \tfrac{d}{dx}\, \log ((E-V(x))^{\f12}) \lb{4.44}
\end{align}
from which \eqref{4.43} follows.
\end{proof}

\begin{proof}[Proof of Theorem~\ref{T4.1}] Let
\begin{equation} \lb{4.45}
Y(x) = \begin{pmatrix}
u_+(x) & u_-(x) \\
u'_+(x) & u'_-(x)
\end{pmatrix}
\end{equation}
Let $T(x,y)$ be the $\binom{u}{u'}$ transfer matrix from $x$ to $y$ and $\ti T(x,y)$ be the
$\binom{a}{b}$ transfer matrix. For $y>N(E)+1$, we have just seen
\begin{align}
\|\ti T(N(E)+1,y)\| &\leq \exp\biggl( \int_{N(E)+1}^\infty \|M(x)\|\, dx \biggr) \notag \\
&= \f{\kappa(\infty, E)}{\kappa(N(E)+1,E)} \lb{4.46}
\end{align}

On the other hand,
\begin{equation} \lb{4.47}
\|Y(y)\|\leq 1+\kappa \leq 2
\end{equation}
for $\kappa$ small while
\begin{equation} \lb{4.48}
\|Y(y)^{-1}\|= \abs{\det(Y)^{-1}}\, \|Y\| \leq \kappa(y)^{-1}
\end{equation}
and
\[
T(x,y) = Y(y) \ti T(x,y) Y(x)^{-1}
\]
so
\begin{equation} \lb{4.49}
\|T(N(E)+1,y)\| \leq \f{2\kappa(\infty, E)}{\kappa (N(E)+1,E)^2}
\end{equation}

Since $E=V(N(E))$,
\begin{equation} \lb{4.50}
\kappa(N(E)+1,E)^2 = V(N(E)) -V(N(E)+1)
\end{equation}
and we have the bound \eqref{4.4} with the error built from $e^{-V(0)}$, $N(E)$, \eqref{4.43}, and
\eqref{4.49}.
\end{proof}

It is interesting that the differential equation methods of this section lead to terms that are
identical to what we found in the discrete case.

\section{Examples} \lb{s6}

We start with the continuum case.

\begin{example}\lb{E6.1}
\begin{equation} \lb{6.1}
V(x)=C_0 x^{-\beta} \qquad \beta <2
\end{equation}

Technically this does not fit into Theorem~\ref{T4.1} since $V(0)=\infty$, but when $\beta <2$, it is
easy to extend the analysis. The spectral measure is $e^{-2Q(E)}\, dE$ where \eqref{4.4} holds.
\begin{equation} \lb{6.2}
N(E)=\biggl(\f{E}{C_0}\biggr)^{-\f{1}{\beta}}
\end{equation}

\begin{align}
V(N(E)) -V(N(E)+1) &\sim V'(N(E)) \notag \\
&\sim N(E)^{-1} V(N(E)) \notag \\
&= EN(E)^{-1} \lb{6.3}
\end{align}
so $h(E)= O(\log (N(E)^2 E^{-\f12})) = O(\log(E))$. On the other hand, letting $y=x/N(E)$,
\begin{align}
g(E) &=\int_0^{N(E)} (V(x)-E)^{\f12}\, dx \lb{6.4} \\
&= N(E) E^{\f12} \int_0^1 (y^{-\beta} -1)^{\f12}\, dy \lb{6.5} \\
&= E^{\f12} N(E) \beta^{-1} \int_0^1 (1-u)^{\f12} u^{\f{1}{\beta}-\f{3}{2}}\, du \notag \\
&= E^{\f12} N(E) \beta^{-1} \, \f{\Gamma(\f32)\Gamma(\f1{\beta}-\f12)}{\Gamma(\f{1}{\beta}+1)} \lb{6.6}
\end{align}
using a $u=y^\beta$ change of variables. Thus,
\begin{equation} \lb{6.7}
g(E) = c_1 C_0^{\f1{\beta}} E^{\f12 -\f1{\beta}} \qquad c_1 = \beta^{-1} \,
\f{\Gamma(\f32)\Gamma(\f1{\beta}-\f12)}{\Gamma(\f1{\beta}+1)}
\end{equation}
Since $\beta <2$, $g(E)\to\infty$ and is much larger than the $\log(E)$ error. $\beta=1$, the Coulomb case,
has $g(E)=C_0 c_1 E^{-\f12}$ and $\beta=\f12$, the quasi-Szeg\H{o} borderline, has $g(E)=C_0^2 c_1 E^{-\f32}$.
We emphasize that $g$ occurs in an exponential, so $w$ is very small near $E=0$.
\qed
\end{example}

\begin{example}
\begin{equation} \lb{6.8}
V(x)=C_0 (x+x_0)^{-\beta} \qquad \beta <2
\end{equation}

We claim that the changes from Example~\ref{E6.1} are small compared to $\log(E)$ errors in $h$; explicitly,
\begin{equation}\lb{6.9}
g(E)=c_1 C_0^{\f1{\beta}} E^{\f12 - \f{1}{\beta}} + O(1) + O(E^{\f12})
\end{equation}
For in this case,
\begin{equation} \lb{6.10}
N(E) = \biggl( \f{E}{C_0}\biggr)^{-\f1{\beta}} -x_0
\end{equation}
and one changes variables to $y=(x+x_0)/(N(E)+x_0)$, so \eqref{6.5} becomes
\begin{equation} \lb{6.11}
g(E) = N(E) E^{\f12} \int_{s(E)}^1 (y^{-\beta} -1)^{\f12}\, dy
\end{equation}
where
\begin{equation}\lb{6.12}
s(E) = y(x=0) = \f{x_0}{N(E)+x_0}
\end{equation}
Then
\begin{align}
N(E) E^{\f12} \int_0^{s(E)} (y^{-\beta}-1)^{\f12} &= N(E) E^{\f12} O(s(E)^{1-\f{\beta}{2}}) \notag \\
&= O(1) \lb{6.13}
\end{align}
by \eqref{6.10} and \eqref{6.12}, so
\begin{align}
g(E) &= c_1 N(E) E^{\f12} + O(1) \notag \\
&= c_1 C_0^{\f{1}{\beta}} E^{\f12 - \f{1}{\beta}} + O(1) +O(E^{\f12}) \lb{6.14}
\end{align}
as claimed.
\qed
\end{example}

Now we turn to the discrete case.

\begin{example}[$=$ \eqref{1.7}] \lb{E6.3}
\begin{equation} \lb{6.15}
a_n\equiv 1 \qquad b_n = -C n^{-\beta}
\end{equation}
Define
\begin{equation} \lb{6.16}
\delta = 2-x \qquad \delta_n = C n^{-\beta} -\delta
\end{equation}
so
\begin{equation} \lb{6.17}
x-b_n = 2+\delta_n
\end{equation}
We have (with $[y]=$ maximal integer $\leq y$)
\begin{equation} \lb{6.18}
N(x) = [(C^{-1}\delta)^{-\f1{\beta}}]
\end{equation}

We have $b_{N+2}-b_{N+1}=O(N^{-\beta-1})$, so the RHS of \eqref{2.5} is of order
$CN(x)^{\beta+2}\delta^{\f12}= O(\delta^{-\f12 -\f2{\beta}})$
and thus, $h(x)=O(\log(2-x))$ and we need to compute $g(x)=\sum_{j=1}^{N(x)}
\gamma_j(x)$ up to $O(\log\delta)$ terms.

We will suppose below that $C\leq 1$ and explain at the end what to change if $C>1$.

Define $c_\ell$ to be the Taylor coefficients in
\begin{equation} \lb{6.19}
\cosh^{-1} (1+\tfrac{z}{2}) =\sqrt{z}\, \sum_{\ell=0}^\infty c_\ell z^\ell
\end{equation}
so, courtesy of Mathematica,
\[
c_0 =1 \qquad c_1 =-\f{1}{24} \qquad c_2 = \f{3}{640} \qquad c_3 = -\f{5}{7168}
\]
and, for example,
\[
c_{20} = 34,461,632,205/12,391,489,651,049,749,040,738,304
\]
(assuming that we managed to copy it without a typo).
Thus,
\begin{equation} \lb{6.20}
g(x) =\sum_{\ell=0}^\infty c_\ell \sum_{j=1}^{N(x)} \delta_j^{\ell+\f12}
\end{equation}

Notice that since $\delta >0$,
\begin{equation} \lb{6.21}
\delta_j \leq C j^{-\beta}
\end{equation}
so, if $\beta (\ell+\f12)>1$, a crude $\delta$-independent bound of
$\sum_{j=1}^{N(x)} \delta_j^{\ell+\f12}$
can be summed independently of $N(x)$.
Moreover, if $F$ is the function in \eqref{6.19}, then
\begin{equation} \lb{6.22}
2 \sqrt{z}\, \f{dF}{dz} =\f{1}{\sqrt{1+\f{z}{4}}}
\end{equation}
so the $c_\ell$ power series has radius of convergence $4$ and so $\sum \abs{c_\ell}<\infty$. Thus, if
\begin{equation} \lb{6.23x}
\ell_0 = [\tfrac{1}{\beta}-\tfrac12] +1
\end{equation}
then
\begin{equation} \lb{6.23}
\sum_{\ell=\ell_0}^\infty\, \abs{c_\ell} \sum_{j=1}^{N(x)} \delta_j^{\ell + \f12} \leq
\biggl(\, \sum_0^\infty \, \abs{c_\ell}\biggr) \sum_{j=1}^\infty j^{-\beta (\ell_0+1)}
\end{equation}
(since $C \leq 1$) so
\begin{equation} \lb{6.24}
\sum_{j=1}^N \gamma_j =\sum_{0\leq\ell\leq \f{1}{\beta}-\f12} c_\ell \sum_{j=1}^N \delta_j^{\ell+\f12} +O(1)
\end{equation}

If $\ell=\f1{\beta}-\f12$ occurs, then
\begin{align}
\sum_{j=1}^N \delta_j^{\f1{\beta}-\f12 +\f12} &= \sum_{j=1}^N \delta_j^{\f1{\beta}} \notag \\
&=\sum_{j=1}^N (C\, \tfrac{1}{j^\beta}-\delta)^{\f1{\beta}} \notag \\
&\leq C^{\f{1}{\beta}} \sum_{j=1}^N j^{-1} \notag \\
&= O(\log N) \lb{6.25}
\end{align}
On the other hand, if $\ell <\f1{\beta}-\f12$, then
\begin{align}
\sum_{j=1}^N \delta_j^{\ell+\f12} &= \sum_{j=1}^N (Cj^{-\beta}-\delta)^{\ell+\f12} \notag \\
&= C^{\ell+\f12} \sum_{j=1}^N ( \tfrac{1}{j^\beta} -\tfrac{1}{N^\beta})^{\ell+\f12} +O(1) \notag \\
&= C^{\ell+\f12} \sum_{j=1}^N j^{-\beta(\ell+\f12)}(1-(\tfrac{j}{N})^\beta)^{\ell+\f12}
  +O(1) \notag \\
&= C^{\ell+\f12} \int_1^N x^{-\beta(\ell+\f12)} (1-(\tfrac{x}{N})^\beta)^{\ell+\f12}
  +O(1) \lb{6.26} \\
&= C^{\ell+\f12} \beta^{-1} N^{1-(\ell+\f12)\beta} \int_{N^{-\beta}}^1 u^{(\f{1}{\beta}-\ell -\f{3}{2})}
  (1-u)^{\ell+\f12}\, du + O(1) \lb{6.27} \\
&= C^{\ell+\f12} \beta^{-1} N^{1-(\ell+\f12)\beta} \int_0^1 u^{(\f{1}{\beta}-\ell -\f{3}{2})}
(1-u)^{\ell+\f12}\, du + O(1) \lb{6.28} \\
&= C^{\ell+\f12} \beta^{-1}\, \f{\Gamma (\ell+\f32)\Gamma (\f{1}{\beta}-\f12-\ell)}{\Gamma(\f{1}{\beta}+1)} \,
  N^{1-(\ell+\f12)\beta} + O(1) \notag
\end{align}

In the above, \eqref{6.26} comes from the fact that the function in the integrand is monotone decreasing,
and if $f(x)$ is monotone, then
\[
f(j) \geq \int_j^{j+1} f(y)\, dy \geq f(j+1)
\]
so
\[
\sum_{j=1}^{N-1} f(j) \geq \int_1^N f(y)\, dy \geq \sum_{j=2}^N f(j)
\]
and
\begin{equation} \lb{6.30}
\biggl| \int_1^N f(y)\, dy -\sum_{j=1}^N f(j)\biggr| \leq f(1)
\end{equation}
\eqref{6.27} is the change of variables $u=(\f{x}{N})^\beta$. Finally, \eqref{6.28} comes from
the same cancellation that occurred in \eqref{6.13}.

Since $\abs{N-C^{\f{1}{\beta}} \delta^{-\f{1}{\beta}}}\leq 1$ and $0<1-(\ell +\f12)\beta < 1$,
\begin{equation} \lb{6.31}
N^{1-(\ell+\f12)\beta} = (C^{\f{1}{\beta}} \delta^{-\f{1}{\beta}})^{1-(\ell+\f12)\beta} + o(1)
\end{equation}
Thus, we find
\begin{equation} \lb{6.32}
Q(x) =\beta^{-1} C^{\f{1}{\beta}} \sum_{0\leq \ell < (\f{1}{\beta}-\f12)} c_\ell \,
\f{\Gamma (\ell + \f32)\Gamma(\f{1}{\beta}-\f12 -\ell)}{\Gamma(\f{1}{\beta}+1)}\,
\delta^{-\f{1}{\beta}+\ell+\f12} + O(\log\delta)
\end{equation}

If $C >1$, we should not expand the power series of $\cosh^{-1}$ for small $j$ (actually, as noted,
the power series has radius of convergence $4$ so we need only worry if $C \geq 4$). Instead, we do
not expand for those $j$ with $C j^{-\beta} >1$. That is only finitely many terms, so it adds
$O(1)$ errors to $\sum_1^N \gamma_j(x)$. We add back these small $j$ terms to \eqref{6.24}, again
making $O(1)$ errors. The final result does not change.
\qed
\end{example}

Finally, we will explore examples that lead to $Q$'s roughly of the type \eqref{1.10} to link to
work of Levin--Lubinsky \cite{LL07}. We suppose
\begin{equation} \lb{6.33}
a_n =1-f(\log(n+1))
\end{equation}
where the $f$'s we have in mind are typically
\begin{equation} \lb{6.34}
f(x)=(1+x)^{-\alpha}
\end{equation}
or
\begin{equation} \lb{6.35}
f(x)=\log_k (x+c_k)
\end{equation}
an iterated log (where $c_k$ is chosen to keep all $\log$'s that enter positive). We will need

\begin{proposition}\lb{P6.4} Let $f$ be defined and $C^2$ on $[\log 2, \infty)$ and obey
\begin{alignat}{2}
&\text{\rm{(i)}} \qquad && f(x) >0, \quad f'(x) <0, \quad f''(x) >0 \lb{6.36} \\
&\text{\rm{(ii)}} \qquad && \lim_{n\to\infty}\, f(n)=0 \lb{6.37} \\
&\text{\rm{(iii)}} \qquad && \lim_{N\to\infty}\, N^\veps (-f'(\log N))^{\f12}=\infty \lb{6.38} \\
&\text{\rm{(iv)}} \qquad && \lim_{\veps\downarrow 0} \, \biggl( \limsup_{k\to\infty}\,
\f{-f'((1-\veps)k)}{-f'(k)}\biggr) =1 \lb{6.39}
\end{alignat}
Let
\begin{equation} \lb{6.40}
S_N =\sum_{j=2}^N \sqrt{f(\log j)-f(\log N)}
\end{equation}
Then
\begin{equation} \lb{6.41}
\lim_{N\to\infty}\, \f{S_N}{N(-f'(\log N))^{\f12}} = \f{\sqrt{\pi}}{2}
\end{equation}
\end{proposition}

\begin{remark} It is easy to see that if $f(x)=e^{-kx}$ (i.e., $f(\log (n+1))\sim (n+1)^{-k}$),
then \eqref{6.41} fails. In this case, both \eqref{6.38} and \eqref{6.39} fail, but they hold
for the $f$'s of \eqref{6.34} \and \eqref{6.35}.
\end{remark}

\begin{proof} Since $(-f')' <0$ and if $x<y$,
\begin{equation} \lb{6.42}
f(x)-f(y) = \int_x^y (-f'(s))\, ds
\end{equation}
we have,
\begin{equation} \lb{6.43}
(y-x)(-f'(y))\leq f(x)-f(y)\leq (y-x)(-f'(x))
\end{equation}

We thus get a lower bound
\begin{equation} \lb{6.44}
f(\log j) -f(\log N)\geq (-f'(\log N))(-\log (\tfrac{j}{N}))
\end{equation}
so
\begin{equation} \lb{6.45}
S_N \geq N (-f'(\log N))^{\f12} \sum_{j=2}^N \tfrac{1}{N} (-\log (\tfrac{j}{N}))^{\f12}
\end{equation}
As $N\to\infty$, the sum converges to $\int_0^1 (-\log(x))^{\f12}\, dx =\f{\sqrt{\pi}}{2}$
(courtesy of Mathematica). Thus,
\begin{equation} \lb{6.46}
\liminf (\text{LHS of \eqref{6.41}})\geq \f{\sqrt{\pi}}{2}
\end{equation}

For the upper bound, fix $\veps >0$ and break $S_N=S_N^{(1)} + S_N^{(2)}$ where $S_N^{(1)}$ has
$j\leq N^{1-\veps}$ and $S_N^{(2)}$ has $j>N^{1-\veps}$. Clearly,
\begin{equation} \lb{6.47}
S_N^{(1)} \leq f(\log 2)N^{1-\veps}
\end{equation}
so, by hypothesis \eqref{6.38}, it contributes $0$ to the ratio in \eqref{6.41} as $N\to\infty$.

For $S_N^{(2)}$, we use the upper bound when $j>N^{1-\veps}$
\[
f(\log j)-f(\log N)\leq -f'((1-\veps)\log N) (-\log (\tfrac{j}{N}))
\]
which yields (since the Riemann sum still converges to the integral)
\[
\limsup(\text{LHS of \eqref{6.41}}) \leq \f{\sqrt{\pi}}{2}\, \limsup_{k\to\infty}\,
\biggl(\f{-f'((1-\veps)k)}{-f'(k)}\biggr)^{\f12}
\]
Since $\veps$ is arbitrary, we can use \eqref{6.39} to complete the proof of \eqref{6.41}.
\end{proof}

\begin{example}\lb{E6.5} Let $a_n$ have the form \eqref{6.31} where $f$ obeys all the hypotheses
of Proposition~\ref{P6.4}. By \eqref{3.2} and \eqref{3.3}, $N(x)$ roughly solves
\begin{equation} \lb{6.48}
\f{x}{1-f(\log (N+1))} =2
\end{equation}
namely,
\begin{equation} \lb{6.49}
N(x) =[\exp(f^{-1} (1-\tfrac{x}{2}))] - 1
\end{equation}

For example, if $f$ is \eqref{6.34}, then
\begin{equation} \lb{6.50}
N(x)=[\exp((1-\tfrac{x}{2})^{-\alpha}-1)] - 1
\end{equation}

Next, define $z$ by $\f{x}{2a}=1+\f{z}{2}$, namely,
\begin{equation} \lb{6.51}
z=\tfrac{x}{a} - 2
\end{equation}
where $\tfrac{x}{a} > 2$. Writing $x=2-\delta$ and $a=1-f$, we see
\begin{equation} \lb{6.52}
z=-\delta + 2f + O(f^2) + O(f\delta)
\end{equation}

Taking into account that $N(x)$ is such that
\[
2f(\log (N+2)) \leq \delta \leq 2f(\log (N+1))
\]
and that \eqref{6.19} says
\[
\cosh^{-1} (\tfrac{x}{2a})=\sqrt{z} + O(z^{\f32})
\]
we see that
\[
\gamma_j(x) = \sqrt{2f(\log (j+1)) -\delta} + O(f^{\f32}) + O(f^{\f12}\delta)
\]
and thus
\[
g(x)=\sum_{j=1}^{N(x)} \gamma_j(x)
\]
is asymptotically the same as $\sqrt{2}\,S_N$. Thus,
\begin{equation} \lb{6.48x}
\abs{Q(x)-g(x)}\leq h(x)
\end{equation}
where
\begin{equation} \lb{6.49x}
g(x) =\sqrt{\frac{\pi}{2}}\, N(x) (-f'(\log N(x)))^{\f12} (1+o(1))
\end{equation}
and
\[
h(x)=O(\log N(x)) + O(\log (1-\tfrac{2}{x}))
\]

$N(x)$ is huge, so while $\log N(x)\sim (1-\f{x}{2})^{-\alpha}$ in case \eqref{6.34}, 
it is still small relative to $g(x)$.

The reader may be puzzled in comparing our results with those of Levin--Lubinsky \cite{LL07}. 
They have no $\sqrt{\tfrac{\pi}{2}}$ and their relations (after making the modifications from 
$[-1,1]$ to $[-2,2]$) suggest
\begin{equation} \lb{6.50x}
1-a_n = (\log n)^{-\f12} (1+o(1))
\end{equation}
should correspond to
\begin{equation} \lb{6.51x}
Q(x) =\exp ((1-\tfrac{x}{2})^{-\alpha})
\end{equation}
so there is no sign of $(-f'(\log N(x)))^{\f12}$ either.

The mystery is solved by the fact that multiple $Q$'s lead to the same leading asymptotics 
for $a_n$. In their scheme, after corrections to move to $[-2,2]$, leading asymptotics for 
$f$ are given by
\begin{equation} \lb{6.52x}
n=Q(1-2(f(n)(1+o(1))))
\end{equation}
If
\begin{equation} \lb{6.53}
Q(x)=e^{1/(1-\f{x}{2})}
\end{equation}
then
\begin{equation} \lb{6.54}
n = \exp((f(n))^{-1})
\end{equation}
solved by
\begin{equation} \lb{6.55}
f(n) = \tfrac{1}{\log n}\, (1+o(1))
\end{equation}

Changing \eqref{6.53} to
\[
Q(x)=\tfrac{\pi}{2} (1-\tfrac{x}{2}) \exp ((1-\tfrac{x}{2})^{-1})
\]
is solved by
\[
f(n) = 1/(\log (\tfrac{2n}{\pi} \log n) + O(\log \log n))
\]
Since
\[
\log \tfrac{2n}{\pi}\, \log n = \log n + \log_2 n + \log (\tfrac{2}{\pi})
\]
\eqref{6.55} still holds!
\qed
\end{example}

\bigskip

\end{document}